\documentclass[10pt,reqno]{amsart}

\usepackage[utf8]{inputenc}
\usepackage{enumitem}
\usepackage{amsmath,amsthm,amssymb,color, tikz}
\usepackage[all]{xy}

\usetikzlibrary{matrix,arrows,positioning,calc}
\usepackage{hyperref}
\hypersetup{
    colorlinks=true,
    citecolor=blue,
    linkcolor=red,
    hyperindex=true,
    linktocpage=true,
    filecolor=magenta,      
    urlcolor=cyan,
}
\theoremstyle{plain}
\newtheorem{definition}{Definition} 
\newtheorem{theorem}[definition]{Theorem}
\newtheorem{lemma}[definition]{Lemma}
\newtheorem{corollary}[definition]{Corollary}
\newtheorem{proposition}[definition]{Proposition}

\newtheorem{question}[definition]{Question}

\theoremstyle{definition}
\newtheorem{example}[definition]{Example}

\newtheorem{remark}[definition]{Remark}

\numberwithin{equation}{section}

\newcommand{\A}{\mathcal A}

\newcommand{\C}{\mathcal C}

\newcommand{\F}{\mathcal F}
\newcommand{\T}{\mathcal T}

\newcommand{\N}{\mathbb N}
\newcommand{\Z}{\mathbb Z}

\newcommand{\Hom}{{\rm Hom}}
\newcommand{\Ext}{{\rm Ext}}
\newcommand{\Tor}{{\rm Tor}}
\newcommand{\Rmod}{R\text{-}{\rm Mod}}

\newcommand{\ses}[3]{0 \rightarrow {#1} \rightarrow {#2} \rightarrow {#3} \rightarrow 0 } 

\newcommand{\kker}[1]{\ensuremath{{{\rm Ker}\left(#1\right)}}}
\newcommand{\coker}[1]{\ensuremath{{{\rm Coker}\left(#1\right)}}}
\newcommand{\im}[1]{\ensuremath{{{\rm Im}\left(#1\right)}}}

\newcommand{\FP}[1]{\mathcal{FP}_{#1}}
\newcommand{\FPinj}[1]{\mathcal{FP}_{#1}\text{-}{\rm Inj}}
\newcommand{\FPflat}[1]{\mathcal{FP}_{#1}\text{-}{\rm Flat}}
\newcommand{\Coh}[1]{{#1}\text{-$\mathit{Coh}$}}
\newcommand{\Her}[1]{{#1}\text{-$\mathit{Her}$}}

\newcommand{\pd}[1]{\text{pd}{(#1)}}
\newcommand{\wdim}[1]{\text{wd}{(#1)}}

\newcommand{\supp}[1]{\text{supp}{(#1)}}

\newcommand{\FPnid}[1]{\text{FP}_n\text{-id}(#1)}
\newcommand{\FPnfd}[1]{\text{FP}_n\text{-fd}(#1)}

\begin{document}

\title{Torsion pairs over $n$-Hereditary rings}

\author{Daniel Bravo}
 \address{
 Instituto de Ciencias F\'isicas y Matem\'aticas \\
 Facultad de Ciencias \\ 
 Universidad Austral de Chile \\
 Valdivia, Chile}
 \email{daniel.bravo@uach.cl}
 \thanks{The first author was partially supported by CONICYT + FONDECYT/Regular + 1180888.}


 \author{Carlos E. Parra}
 \email{carlos.parra@uach.cl}
  \thanks{The second author was partially supported by CONICYT + FONDECYT/Iniciaci\'on + 11160078.}


\begin{abstract} We study the notions of $n$-hereditary rings and its connection to the classes of finitely $n$-presented modules, FP$_n$-injective modules,  FP$_n$-flat modules and $n$-coherent rings. We give characterizations of $n$-hereditary rings in terms of quotients of injective modules and submodules of flat modules, and a characterization of $n$-coherent using an injective cogenerator of the category of modules. We show  two torsion pairs with respect to the FP$_n$-injective modules and the FP$_n$-flat modules over  $n$-hereditary rings. We also provide an example of a B\'ezout ring which is 2-hereditary, but  not 1-hereditary, such that the torsion pairs over this ring are not trivial.
\end{abstract}




\date{\today}

\maketitle


\tableofcontents

\section*{Introduction}\label{introduction}

The notion of torsion pair was introduced in the sixties by S. Dickson  in the setting of abelian categories, generalizing the classical notion of torsion pairs for abelian groups; see \cite{Dick}. Since then, the theory of torsion pairs has been greatly developed and  many applications have been given to areas such as  representation theory of Artin algebras, homological algebra, non commutative localization theory, and tilting theory to mention a few; see \cite{Assem-Saorin}, \cite{Bazzoni-Her},   \cite{Colpi-Trilifaj}, \cite{Colpi}, \cite{Gobel}, \cite{HRS}, \cite{rings-of-quot}, . The importance of torsion pairs is highlighted by the theorem of Popescu and Gabriel \cite[X, \S 4]{rings-of-quot} which reduces the  theory of Grothendieck categories  to the study of categories of modules of quotients by a hereditary torsion pairs. All this have made the theory of torsion pairs a valuable toolkit and an active research area on its own; see \cite{BraPa}, \cite{CGM}, \cite{Hrbek}, \cite{Parra-Saorin}.

Recently the classes of FP$_n$-injective modules and FP$_n$-flat modules have been studied in detail, generalized to chain complexes and applications have been given to cotorsion pairs, duality pairs, and model categories; see \cite{BGH}, \cite{BraPe}, \cite{Pe-Zhao}. In particular, some of those results showed that over certain generalization of coherent rings, namely $n$-coherent rings, the cotorsion pairs are well behaved. In this sense, it seems natural to investigate whether these classes of modules also fit in the theory of torsion pairs. Alternatively, one could ask  if there are any conditions required on the ring such that any such torsion pair exist. Following the known facts that the classes of FP$_n$-injective and FP$_n$-flat modules are closed under products, summands and extensions, when $n>1$, it remains as the main obstacle for these classes of modules to form part of a torsion pair, to be also  closed by either quotients or submodules (indeed, we already know these classes are closed under pure submodules and pure quotients). 

A classical result from H. Cartan and S. Eilenberg shows that over hereditary rings the class of injective modules is closed under quotients; see \cite{C-E}.  C. Meggiben observed a slightly more general result, namely, that over semi-herditary rings the class of FP-injective modules is also closed under quotients; see \cite{megibben}. In this way, we are motivated to investigate such closure properties for FP$_n$-injective modules over a generalized version of semi-hereditary rings. Thus, we introduce the concept of $n$-hereditary rings to reach this goal and to investigate how far  the classes of  FP$_n$-injective modules and FP$_n$-flat modules are from being part of a torsion pair. This motivation naturally leads to investigate any relevant properties of these rings and its connections to these classes of modules. In fact, we show that over $n$-hereditary rings, the class of FP$_n$-injective modules is the torsion class of a torsion pair, and that the class of FP$_n$-flat modules is the torsion-free class of a torsion pair. We are also able to provide an example of  2-hereditary ring, that  is not 1-hereditary (or semi-hereditary), which shows that the torsion pairs in question are in fact non trivial.

This article is organized as follows. Section \ref{FPn} describes the class of finitely $n$-presented modules, the basic object of our study; collect some of its properties  and describe a relevant property for the sections to follow that also doesn't seem to be available previously in the literature. Section \ref{sec-n-her-n-coh} introduces the notion of $n$-hereditary rings and investigate its relation with the class of the finitely $n$-presented modules and $n$-coherent rings.   The notion of FP$_n$-injective and FP$_n$-flat modules is recalled in Section \ref{Relative-Injective-and-Flat}, where the key relations of these two classes of modules with $n$-hereditary rings are also investigated. We also show new a characterization of $n$-coherent rings in terms of closure properties of a certain collection of injective modules and the class of FP$_n$-injective modules. These results are used in Section \ref{sec-torsion} to establish the two torsion pairs over $n$-hereditary rings, and its connection with (co)tilting classes; as an application of these results, we get from these torsion pairs a far from obvious results about the class of FP$_n$-flat modules. Finally, two appendixes sections are added; in the first one, we show that the mentioned 2-hereditary ring example is a 2-coherent, B\'ezout ring, but not 1-hereditary, and the necessary properties which give that the torsion pairs from the previous section are not trivial. The second appendix is about a property of Grothendieck categories that when applied to the setting of modules, gives the key result for the new characterization of $n$-coherent rings.

Throughout this paper, {$R$ denotes an associative ring with unit}, $\Rmod$  denotes category of left $R$-modules, and unless otherwise noted, the expression $R$-module mean left $R$-module.


\section{Finitely $n$-presented modules} \label{FPn}

Let $n \geq 0$ be an integer. An $R$-module $M$ is said to be \textit{finitely $n$-presented}, if  there is an exact sequence 
\[ 
F_n \to F_{n-1} \rightarrow \cdots \rightarrow F_1 \rightarrow F_0 \rightarrow M \rightarrow 0,
\]
where the modules $F_i$ are finitely generated and free (or projective) modules, for every $0 \leq i \leq n$. Denote by \textit{$\FP{n}$} the class of all finitely $n$-presented modules. Thus $\FP{0}$ is the class of finitely generated modules, and $\FP{1}$ is the class of finitely presented modules. For convenience, we let $\FP{-1}$ be the whole class of $R$-modules. Also we consider the class, \textit{$\FP{\infty}$}, of the \textit{finitely $\infty$-presented} modules,  formed by modules that posses a resolution by finitely generated free (or projective) modules.
Note that the class $\FP{\infty}$ is not empty, since any finitely generated projective module is finitely $\infty$-presented. 

We immediately observe the following descending chain of inclusions:
\begin{equation}\label{chain-of-FPn}
\FP{0} \supseteq \FP{1} \supseteq \cdots \supseteq \FP{n} \supseteq \FP{n+1} \supseteq \cdots \supseteq \FP{\infty}.
\end{equation}

We include two examples of rings to show how the chain \eqref{chain-of-FPn} behaves; for more details about these two examples we refer the reader to \cite{BraPe}.

\begin{example} \label{2-coherent-ring-example}
Let $k$ be a field and $R$ be the following polynomial ring:
\[
R := k[x_1,x_2,x_3,\ldots] / ( x_i x_j)_{i,j \geq 1}.
\]
In this ring every finitely $2$-presented module is finitely generated free.  Also we can quickly check that $R/(x_1) \in \FP{1} \setminus \FP{2}$ and that $(x_1) \in \FP{0} \setminus \FP{1}$.  Thus we have that the chain of inclusions in \eqref{chain-of-FPn} collapses at $2$:
\[
\FP{0} \supset \FP{1} \supset \FP{2}  = \FP{n} = \FP{\infty}.
\]
\end{example}

The next example shows that \eqref{chain-of-FPn} may never collapse. 

\begin{example}  \label{strictly-infty-ring}
Let $k$ be a field and consider the following ring:
\[
R := {k[\ldots,x_{3},x_2,x_1,y_1,y_2,y_3,\ldots]} / {(x_{j+1} x_{j},  x_1 y_1, y_1 y_i)_{i,j \geq 1}}
\]
Then $(y_1) \in \FP{0} \setminus \FP{1}$, and $(x_i) \in \FP{i} \setminus \FP{i+1}$ for $i \geq 1$. Hence in this case, the chain in \eqref{chain-of-FPn} is strict at every level.
\end{example}

Several results about $\FP{n}$ and $\FP{\infty}$ are collected in \cite{BraPe}. We include here the following results. 

\begin{proposition} \label{properties_FPn} 
Let $n \geq 0$. $\FP{n}$ is closed under cokernels of monomorphisms, extensions, and direct summands.
\end{proposition}

The following results also appear in \cite{Brown-Cohomology-of-groups}, and \cite{bieri}.
\begin{theorem} \label{FPn_Ext_Tor}
The following conditions are equivalent for every left $R$-module $M$ and every $n \geq 0$:
\begin{enumerate}
\item $M \in \FP{n+1}$.

\item $\Ext^i_R(M,-)$ commutes with direct limits for all $0 \leq i \leq n$.

\item $\Tor_i^R(-,M)$ commutes with direct products for all $0 \leq i \leq n$. 
\end{enumerate} 
\end{theorem}

The class  $\FP{\infty}$ has all the properties from the previous proposition and one more, as indicated in the following result.

\begin{theorem} \label{FPinfty-is-thick}
 $\FP{\infty}$ is closed under kernels of epimorphisms.
\end{theorem}

\begin{remark} \label{FPn-set}
For any finitely generated module  $M$, we have that $\text{Card}(M) \leq \max \{ \aleph_0, \text{Card}(R) \}$. Hence 
  we can choose $\mathcal{S}_n$,  a \emph{set} of representatives of finitely generated modules in $\FP{n}$, such that every module in $\FP{n}$ is isomorphic to a module in $\mathcal{S}_n$.
\end{remark}

In the next two sections, we will use the class of finitely $n$-presented modules to describe two types of rings.


\section{$n$-Hereditary rings and $n$-coherent rings} \label{sec-n-her-n-coh}

As defined in \cite{C-E}, recall that a ring is said to be left \textit{hereditary} if every left ideal is a projective module. This is also equivalent to saying that every submodule of a projective left module is also a projective module, or that every quotient (homomorphic image) of an injective left module is injective. A bit more general are left \textit{semi-hereditary} rings; that is, rings such that every finitely generated left ideal is projective. This is equivalent to saying that every finitely generated submodule of a projective left module is also a projective module, or that every quotient (homomorphic image) of a FP-injective left module is FP-injective {\cite[Theorem 2]{megibben}}. From these observations we define the following:

\begin{definition} \label{def-n-her}
A ring is said to be left $n$-hereditary if every finitely $(n-1)$-presented  submodule of a finitely generated  projective left module is also a projective module.
\end{definition}

This way a left $1$-hereditary ring is the same as a left semi-hereditary ring, and if we allow for \textit{finitely $(-1)$-presented} modules to be any module, then left $0$-hereditary rings coincide with left hereditary. From now on, and unless otherwise noted, the term $n$-hereditary ring will mean left $n$-hereditary ring.

A  characterization of $n$-hereditary rings can be given in terms of the class of finitely $n$-presented modules, $\pd{M}$,  the projective dimension of an $R$-module $M$, $\wdim{M}$, the weak dimension of the $R$-module $M$.

\begin{lemma} \label{proj-dim-FPn-less-than-1}
Let $n \geq 2$. The following statements are equivalent.
\begin{enumerate}
\item $R$ is a  $n$-hereditary ring.  \label{dim-characterization-of-n-her}
\item $\pd{M} \leq 1$ for all $M \in \FP{n}$. \label{proj-dim-characterization}
\item $\wdim{M} \leq 1$ for all $M \in \FP{n}$. \label{weak-dim-characterization}
\end{enumerate}
Furthermore if $n = 1$, then we have the following statement: $R$ is  $1$-hereditary if and only if $\pd{M}\leq 1$ for all $M \in \FP{1}$.
\end{lemma}

\begin{proof}
\eqref{dim-characterization-of-n-her} $\implies$ \eqref{proj-dim-characterization}: 
Suppose $R$ is an $n$-herditary ring and let $M \in \FP{n}$. Consider the following short exact sequence:
\[
\ses{\Omega M}{R^k}{M}.
\]
Since $\Omega M \in \FP{n-1}$ and $R^k$ is a finitely generated and projective module, then $\Omega M$ is also projective. Thus $\pd{M} \leq 1$.

\eqref{proj-dim-characterization} $\implies$ \eqref{weak-dim-characterization}: 
This is clear, since  $\wdim{M} \leq \pd{M}$ for any $M \in \Rmod$.

\eqref{weak-dim-characterization} $\implies$ \eqref{dim-characterization-of-n-her}: 
Suppose that $N \in \FP{n-1}$ is a submodule of a finitely generated projective module $P$. Then $P/N \in \FP{n}$, and so $\wdim{P/N} \leq 1$; hence $N$ is flat. But $n\geq 2$, so $N$ is at least in $\FP{1}$.  Thus, {\cite[Proposition 3.2.12]{E-J}} gives that $N$ is projective, showing that  $R$ is  $n$-hereditary.

Finally, for the last statement, observe that  if $I$ is a finitely generated ideal, then $R/I \in \FP{1}$ and so $I$ is projective; thus $R$ is semi-hereditary, or 1-hereditary. For the converse, note that the proof  of \eqref{dim-characterization-of-n-her} $\implies$ \eqref{proj-dim-characterization} also applies. 
\end{proof}

Using this last characterization,  we give an example of a (commutative) $2$-hereditary ring, that is not $1$-hereditary. 
\begin{example} \label{2-her-ring}

Let $R = \Z \oplus \bigoplus_{i \geq 1} \Z / 2\Z$ with addition defined component wise,  and multiplication given by 
\[
(m,a) \cdot (n,b) = (mn, mb + na + ab)
\]
where $m,n \in \Z$, $a,b \in \bigoplus_{i \geq 1} \Z / 2\Z$ and $m \cdot a = (m a_1,m a_2,m a_3, \ldots)$.

From  {\cite[Example 1.3(b)]{Vasconcelos}}, we have that $\text{gl.wd}(R) \leq 1$; here $\text{gl.wd}(R)$ corresponds to the weak global dimension of the ring $R$. So in particular,   $\text{wd}(M)\leq 1$, for any $M \in \FP{n}$ and  any $n\geq 2$. Thus by Lemma \ref{proj-dim-FPn-less-than-1} we have that $R$ is $2$-hereditary.  In Proposition \eqref{not-1-her} of the Appendix \ref{proof-of-the-ring} section, we show that this ring is not a semi-hereditary ring, or 1-hereditary. 
\end{example}

Another immediate observation is that any ring such that $\text{gl.wd}(R) \leq 1$ is $2$-hereditary. The same proof as in the previous example works, and indeed we note that it is also $n$-hereditary for all $n \geq 2$, but this is always the case as we observe next.

As a consequence of  Definition \ref{def-n-her} and the chain \eqref{chain-of-FPn}, if $R$ is $n$-hereditary, then $R$ is also $k$-hereditary for all $k \geq n$.  Hence, if \Her{n} denotes the collection of all $n$-hereditary rings, then we get the following chain:
\begin{equation} \label{chain-for-n-her}
\Her{0} \subset \Her{1} \subset \Her{2} \subset \cdots \subset \Her{n} \subset \cdots \subset \Her{\infty},
\end{equation}
where $\Her{\infty}$ is the corresponding definition using the class $\FP{\infty}$. 

We have given examples of rings that are 0-hereditary, 1-hereditary and 2-hereditary. Next we show how to get non-commutative $n$-hereditary rings from a given $n$-hereditary ring. 

\begin{example} 
Let $n$ be an integer such that $n \geq 1$. Then $R$ is an $n$-hereditary ring if and only if $M_n(R)$ is an $n$-hereditary ring.
Indeed, the key idea is to observe that there is an equivalence of between the module category of $R$-mod and $M_n(R)$-mod (see  \cite[Theorem 17.20]{Lam}) and that every such equivalence preserve homological properties such as projective modules and exact sequences, and thus finitely $n$-presented modules. 
\end{example}
%
%
%
%
%
%
We also give an example which shows that there are rings that are not $\Her{\infty}$. But to do that we need the following result.

\begin{corollary} \label{charac-of-n-her-by-ideals}
Let $n \geq 0$. If $R$ is $n$-hereditary ring then every ideal $I \in \FP{n-1}$ is projective.
\end{corollary}

\begin{proof}
The cases $n=0$ and $n=1$ are known and can be found in \cite{C-E} and \cite{megibben} respectively;  recall that we are allowing $\FP{-1} = \Rmod$. 

Now let $n>1$ and suppose $R$ is $n$-hereditary. Then for any ideal $I \in \FP{n-1}$, we have that $R/I \in \FP{n}$. Applying Lemma \ref{proj-dim-FPn-less-than-1} gives us that $\pd{R/I} \leq 1$, which implies that $I$ is projective.

\end{proof}

\begin{example}
Let $R=\Z/4\Z$. Then the ideal $I=2\Z/4\Z$ of $R$ is in $\FP{\infty}$ as the resolution $\cdots \to \Z/4\Z \xrightarrow{\cdot 2} \Z/4\Z \xrightarrow{\cdot 2} I \xrightarrow{} 0$  shows. However, the ideal $I$ is not projective, and thus by Corollary \ref{charac-of-n-her-by-ideals}, we see that $R$ is not $\Her{n}$ for any $n \geq 0$.
\end{example}

We still would like to find an explicit example of a ring that is $\Her{\infty}$, but not $\Her{n}$ for any $n\geq 0$.

Corollary \ref{charac-of-n-her-by-ideals} can be thought as a first step in a characterization of $n$-hereditary rings in terms of its ideals. Since as noted in the begining of this section the reciprocal of this corollary works for $n=0,1$, thus it is natural to ask the following.
\begin{question}
Is the reciprocal of Corollary \ref{charac-of-n-her-by-ideals} true? In other words, if every ideal $I \in \FP{n-1}$ is projective, then is it true that $R$ is $n$-hereditary?  
\end{question}

Also related to the idea of finitely $n$-presented modules, is the notion of $n$-coherent rings, which generalizes that of coherent rings. Recall that a ring $R$ is said to be left \textit{coherent} if every finitely generated left ideal of $R$ also is finitely presented.  Another equivalent definition for coherent ring is as follows: $R$  is a left coherent ring if, and only if,  every module in $\FP{1}$ is also in $\FP{2}$. 

\begin{definition}\label{def-n-coherent} 
A ring $R$ is \textit{left $n$-coherent} if $\FP{n} \subseteq \FP{n+1}$.
\end{definition}

Similarly, from now on, and unless otherwise noted, the term $n$-coherent ring will refer to left $n$-coherent ring. 
Thus coherent rings are just 1-coherent rings, and 0-coherent rings coincide with Noetherian rings.   The rings in Example \ref{2-coherent-ring-example} and Example \ref{2-her-ring} are a 2-coherent ring, however a proof of the latter is given in Proposition \ref{2-coherent-Bezout}. The following is an example of an $n$-coherent ring, for any $n \geq 1$.
\begin{example}\label{n-coh-ring}
Let $S=k[[\partial_1,\partial_2,\ldots,\partial_n]]$ be the power series over a collection of $n$ variables, and consider the $S$-module $M = k[x_1,x_2,\ldots,x_n]$ with a linearly extended $S$-action given by $\partial_i x_j = \delta_{ij}$. Consider now the ring $R = S \ltimes M$ given by the trivial extension of the ring $S$ by the $S$-module $M$, defined over the set  $R = \{ (s,m) : s \in S \text{ and } m \in M \}$ and with product given by $(s,m) \cdot (s',m') = (ss',sm' + s'm)$. 

Then the ring $R = k[[\partial_1,\partial_2,\ldots,\partial_n]] \ltimes k[x_1,x_2,\ldots,x_n]$ is $n$-coherent. This is a concrete example of a more general result of J. Roos \cite[Theorem A]{Roos}.
\end{example}

\begin{remark} Note that if $R$ is $n$-coherent, then it is also $k$-coherent, for all $k \geq n$.
Thus, if $\Coh{n}$  denotes the class of all $n$-coherent rings, and if by convention, we allow any ring  to be $\infty$-coherent, then we obtain the following chain: 
\begin{equation} \label{chain-of-coh}
\Coh{0} \subset \Coh{1} \subset \Coh{2} \subset \cdots \subset \Coh{n} \subset \cdots \subset \Coh{\infty}.
\end{equation}
\end{remark}

Unlike the situation for $n$-hereditary rings, we observe that any ring can be thought as an $\infty$-coherent ring. Furthermore, there are rings that are never $n$-coherent for any $n \geq 0$; see {\cite[Example 1.4]{BraPe}}.

The  following theorem states equivalent conditions for the $n$-coherence of a ring in terms of finitely $n$-presented modules.

\begin{theorem}[{\cite[Theorem 2.4 and Corollary 2.6]{BraPe}}] \label{characterization-of-n-coherent} 
Let $R$ be a ring and $n \geq 0$. The following are equivalent.
\begin{enumerate}
\item $R$ is $n$-coherent. \label{R-n-coh}

\item $\FP{n}$ is closed under kernels of epimorphisms. \label{FPn-thick}

\item $\FP{n} = \FP{\infty}$. \label{FPn=FPinfty} 

\end{enumerate} 
\end{theorem}


To end this section we establish a connection between $n$-hereditary rings and $n$-coherent rings.

\begin{corollary} \label{n-her-implies-n-coh}
Let $n\geq 1$. If $R$ is $n$-hereditary, then it is $n$-coherent.
\end{corollary}
\begin{proof}
Suppose $M \in \FP{n}$. Then from Lemma \ref{proj-dim-FPn-less-than-1} we have that $\pd{M}\leq 1$.  Hence in the short exact sequence 
\[
0 \to \Omega M \to F \to M \to 0,
\]
with $F$ finitely generated and free, we get that the syzygy $\Omega M$ is also finitely generated and projective. Therefore $\Omega M \in \FP{\infty}$, giving us that $M \in \FP{\infty}$. 

\end{proof}
The case  $n=0$, would say that any hereditary ring is Noetherian, but this is not the case. Consider, $R=k\langle x,y\rangle$, the polynomial ring over a field $k$ in two nonconmuting variables; this is an (right and left) hereditary ring, but not Noetherian; see  {\cite[Example 4.12]{rotman}}.


\section{Relative homological algebra over $n$-hereditary rings} \label{Relative-Injective-and-Flat}

Having mentioned the class of finitely $n$-presented module, we discuss the relative homological algebra with respect to $\FP{n}$  and define the corresponding  relative injective modules and relative flat modules.

\begin{definition}  \label{Def-FPn-inj}
Let $R$ be a ring, $M \in \Rmod$ and $n \geq 0$ (including the case $n = \infty$).
\begin{enumerate}
\item \label{FPn-injective} 
We say that  $M$ is \textit{$\text{FP}_n$-injective} if $\Ext^1_R(F, M) = 0$ for all $F \in \FP{n}$. We denote by $\FPinj{n}$ the class of all $\text{FP}_n$-injective modules.

\item \label{FPn-flat} 
We say that  $M$ is \textit{\text{$\text{FP}_n$-flat}} if $\Tor_1^R(F, M) = 0$ for all $F \in \FP{n}$. We denote by $\FPflat{n}$ the class of all $\text{FP}_n$-flat modules.
\end{enumerate}
\end{definition}

With these definitions, $M$ is injective if, and only if, $M$ is $\text{FP}_0$-injective, and $M$ is FP-injective (as introduce by \cite{stenstrom}) if, and only if, $M$ is $FP_1$-injective. The usual flat modules coincide with the FP$_0$-flat modules. Given that any module is the direct limit of finitely $1$-presented modules, and that the functor $\Tor_1(-,M)$ commutes with direct limits, then we have that the FP$_1$-flat modules also coincide with the usual flat modules. 

From the descending chain of inclusions \eqref{chain-of-FPn}, we get the following ascending chains of inclusions: 

\begin{equation}\label{chain-of-FPn-inj}
\FPinj{0} \subseteq \FPinj{1} \subseteq \cdots \subseteq \FPinj{n} \subseteq \cdots \subseteq \FPinj{\infty}
\end{equation}
and 
\begin{equation}\label{chain-of-FPn-flat}
\FPflat{0} = \FPflat{1} \subseteq \cdots \subseteq \FPflat{n} \subseteq \cdots \subseteq \FPflat{\infty}.
\end{equation}

For the rest of this article and motivated by these last chains of inclusion, we focus on the case when $n>1$. The following two result appear in \cite{BraPe} and  list several properties about $\FPinj{n}$ and $\FPflat{n}$

\begin{proposition} \label{closure-properties} 
Let  $n>1$. The classes $\FPinj{n}$ and $\FPflat{n}$ are closed under:
\begin{enumerate}
\item Direct summands and extensions. 
\item Direct products and direct limits. 
\item Pure submodules and pure quotients.
\end{enumerate}

\end{proposition}

Given a left $R$-module $M$, recall that the character module is defined as the right $R$-module $M^+ = \Hom_{\Z}(M,\mathbb{Q} / \Z)$. Similarly, the character module of a right $R$-module $M$ is defined in the same way, and it is a left $R$-module that will be also denoted by $M^+$. The classes $\FPinj{n}$ and $\FPflat{n}$ relate well through the character modules. 

\begin{proposition} \label{FPnflat-iff-FPninj^+} 
Let $n>1$. 
\begin{enumerate}
\item $M \in \FPflat{n}$ if and only if $M^+  \in \FPinj{n}$. 
\item $M^+ \in \FPflat{n}$ if and only if $M  \in \FPinj{n}$.
\end{enumerate}
\end{proposition}

We provide a result regarding lifting properties of $\FPinj{n}$.

\begin{proposition} \label{lift-of-FPn-inj}
Let $n \geq 0$.  $M \in \FPinj{n}$ if and only if for every diagram with $P' \in \FP{n-1}$ and $P$ finitely generated projective module, there is a homomorphism $P \xrightarrow{h} M$ such that $hg=f$.
\[
\begin{tikzpicture}
\node (P') at (0,0) {$P'$};
\node (P) at (1,0) {$P$};
\node (M) at (0,-1) {$M$};
\draw[->] (P') -- (M) node[pos=0.5,left]{\tiny $f$};  
\draw[right hook->] (P') -- (P) node[pos=0.5,above]{\tiny $g$};  
\draw[dashed,->] (P) -- (M) node[pos=0.5,right]{\tiny $h$};  
\end{tikzpicture}
\]
\end{proposition}
\begin{proof}
Suppose that $M$ has this lifting property. Let $F \in \FP{n}$, and consider the short exact sequence $\ses{P'}{P}{F}$, with $P$ finitely generated free and $P' \in \FP{n-1}$.  Then applying $\Hom_R(-,M)$ to this short exact sequence gives that $\Ext^1(F,M)=0$, since $\Hom(P,M) \to \Hom(P',M)$ is an epimorphism. Hence $M \in \FPinj{n}$.

Conversely the argument works similarlly, since for $M \in \FPinj{n}$ we have that $\Ext^1(P/P',M)=0$, given that $P/P' \in \FP{n}$, and so $\Hom(P,M) \to \Hom(P',M)$ is an epimorphism.
\end{proof}

The next definition has recently been introduced by M. P\'erez and T. Zhao in their study of syzygies and further generalizations to chain complexes; see \cite{Pe-Zhao}. Nevertheless we recall these definitions here and  investigate its relation to $n$-hereditary rings.   

\begin{definition}
Let $M \in \Rmod$ and $n\geq 0$. The \emph{FP$_{n}$-injective dimension of $M$}, which will be denote by $\FPnid{M}$, is given by the smallest integer $k\geq 0$ such that $\Ext^{k+1}_R(F,M)=0$ for every $F\in \FP{n}$. 
Similarly,  the \emph{FP$_{n}$-flat dimension of $M$}, which will be denote it by $\FPnfd{n}$, is given by the smallest integer $k\geq 0$ such that $\Tor^{R}_{k+1}(F,M)=0$, for every $F\in \FP{n}$.
\end{definition}

With this definition at hand and motivated by the work of H. Cartan and S. Eilenberg \cite{C-E} and C. Megibben \cite{megibben}, we obtain a similar result regarding when  $\FPinj{n}$ is closed under quotients and also investigate when $\FPflat{n}$ is closed under subobjects.

\begin{theorem}\label{teo. big}
Let $n \geq 1$. Then the following are equivalent.
\begin{enumerate}
\item \label{R-is-n-her-ring} $R$ is an $n$-hereditary ring.
\item \label{quotients-FP_n-inj} Quotients of  FP$_n$-injective modules are  FP$_n$-injective.
\item \label{quotients-injective} Quotients of  injective modules are FP$_n$-injective. 
\item \label{submodule-FP_n-flat} Submodules of  FP$_{n}$-flat modules are FP$_{n}$-flat.
\item \label{submodule-flat} Submodules of flat modules are FP$_n$-flat.
\item \label{FP_n-flat-dim} $\FPnfd{M} \leq 1$, for every $M \in \Rmod$.
\item \label{FP_n-inj-dim}$\FPnid{M} \leq 1$, for every $M \in \Rmod$.
\end{enumerate}
Furthermore, for $n=0$ we have that  \eqref{R-is-n-her-ring} $\iff$ \eqref{quotients-FP_n-inj} $\iff$ \eqref{quotients-injective} $\iff$ \eqref{FP_n-inj-dim} and that \eqref{submodule-FP_n-flat} $\iff$ \eqref{submodule-flat} $\iff$ \eqref{FP_n-flat-dim}. 
\end{theorem}
\begin{proof}
\eqref{R-is-n-her-ring} $\implies$ \eqref{quotients-FP_n-inj}. Suppose $M \in \FPinj{n}$ and consider a short exact sequence with $M$ in the middle:
\[
0 \to K \to M \to M' \to 0.
\] 
For each $F \in \FP{n}$, we consider  the following exact sequence
\[
\cdots \to \Ext^1_R(F,M) \to \Ext^1_R(F,M') \to \Ext^2_R(F,K) \to \cdots.
\]
If $R$ is $n$-hereditary, then $\pd{F} \leq 1$, and therefore $\Ext^2_R(F,K) = 0$. Also since $M \in \FPinj{n}$, then $\Ext^1_R(F,M) =0$. This gives us that $\Ext^1_R(F,M') =0$ for any $F \in \FP{n}$.

\eqref{quotients-FP_n-inj} $\implies$ \eqref{quotients-injective}. This implication is easy since every injective module is in $\FPinj{n}$. 

\eqref{quotients-injective} $\implies$ \eqref{FP_n-inj-dim}. Let $M \in \Rmod$ and denote its injective envelope by $E(M)$. Then, the short exact sequence 
\[
0 \to M \to E(M) \to E(M)/N \to 0
\]
gives us that  $\Ext^{2}_R(F,M)=0$, for all $F\in \FP{n}$, since by hypothesis $E(M)/M\in \FPinj{n}$.

\eqref{FP_n-inj-dim} $\implies$ \eqref{R-is-n-her-ring}. Suppose that $N \in \FP{n-1}$ is a submodule of a finitely generated projective module $P$. Hence from the short exact sequence
\[
0 \rightarrow N \rightarrow P \rightarrow P/N \rightarrow 0,
\]
we have that $P/N \in \FP{n}$. Now, to this short exact sequence, apply the functor $\Hom_{R}(-,M)$,  for any $M \in \Rmod$, and obtain the following exact sequence in $\Rmod$:
\[
\cdots \to \Ext^{1}_R(P,M)=0 \to \Ext^{1}_R(N,M) \to \Ext^{2}_R(P/N,M) \to \cdots.
\]
By hypothesis, we have that $\FPnid{M} \leq 1$, hence $\Ext^{2}_R(P/N,M)=0$. Since $P$ is a projective module, we get that  $\Ext^{1}_R(N,M)=0$ and therefore that $N$ is a projective module also.

\eqref{submodule-FP_n-flat} $\implies$ \eqref{submodule-flat}. Since every flat module is in $\FPflat{n}$, we get this immediately.

\eqref{submodule-flat} $\implies$ \eqref{FP_n-flat-dim}. Let $M \in \Rmod$, and consider a short exact sequence 
\[
0 \rightarrow K \rightarrow \bigoplus_{h \in\Hom_{R}(R,M)} R_h \rightarrow M \rightarrow 0,
\]
where $R_h = R$. By hypothesis, we have that $K\in \FPflat{n}$, and so from this sequence we obtain that $\Tor^{R}_{2}(F,M)=0$, for every $F\in \FP{n}$. 

\eqref{FP_n-flat-dim} $\implies$ \eqref{submodule-FP_n-flat}. Suppose that we have an exact sequence 
\[
0 \to N \to M \to M/N \to 0
\]
with $M \in \FPflat{n}$. By assumption, we know that $\Tor^R_{2}(F,M/N)=0$, for every $F\in \FP{n}$, and so  from this short exact sequence we have $\Tor^R_1(F,N)=0$, for every $F\in \FP{n}$.

For the rest of the proof, we suppose that $n>1$. 

\eqref{submodule-FP_n-flat} $\implies$ \eqref{quotients-FP_n-inj}. Suppose that we have an exact sequence $B \to C \to 0$ with $B \in \FPinj{n}$. Then we get an exact sequence $0 \to C^+ \to B^+$, and Proposition \ref{FPnflat-iff-FPninj^+} tells us that $B^+ \in \FPflat{n}$. From our hypothesis we get  $C^+ \in \FPflat{n}$ and therefore by Proposition \ref{FPnflat-iff-FPninj^+} again, we obtain that $C \in \FPinj{n}$.

\eqref{quotients-FP_n-inj} $\implies$ \eqref{submodule-FP_n-flat}. Dual to the proof of  \eqref{submodule-FP_n-flat} $\implies$ \eqref{quotients-FP_n-inj}.


\end{proof}

We would like to point out a new characterization of $n$-coherent rings in terms of the class $\FPinj{n}$. Indeed, several such characterization already are  given in \cite[Theorem 5.5]{BraPe}. The following result also  appears in that same article, but not precisely in the following format.

\begin{corollary}[{\cite[Lemma 5.2]{BraPe}}] \label{Ext-with-FPn-inj}
$R$ is $n$-coherent if and only if $\Ext^1_R(M,N) = 0$ for all $M \in \FP{n}$ and all $N \in \FPinj{n+1}$.
\end{corollary}

This characterization, given in terms of the class $\FPinj{n}$, makes us wonder if perhaps a similar characterization of $n$-coherent rings can be given just in terms of injective modules. Indeed, this is what we do next, however first we introduce some notation. 

Let $I$ be any injective cogenerator in $\Rmod$ (for example $I = \Hom_{\mathbb Z}(R,\mathbb Q / \mathbb Z)$, which can be thought as a left $R$-module, as explain in the paragraph after Proposition \ref{FPnflat-iff-FPninj^+}) and define the class $\mathcal{I}$ as follows: 
\[
\mathcal{I} = \left\{ \prod_{h \in \Hom_R(M,I)} I_h : \text{where $I_h = I$ and } M \in \Rmod \right\}.
\]
Given a class of modules $\mathcal A$, we denote by $\varinjlim A$ the direct limit closure in $\Rmod$ of the  class $\mathcal A$. We denote by $\Omega^{-i}(\mathcal A)$ the class of all the $i$-th cozysygies of injective corresolutions from objects in $\mathcal A$.

\begin{theorem}\label{n-coherent rings}
Let $R$ be a ring, $n > 0$ and $\mathcal{I}$ as describe above. Then the following are equivalent. 
\begin{enumerate}
\item $R$ is $n$-coherent.   \label{n-coh-lim-1}
\item $\Ext^{n}_R(M,X) = 0$ for all $M\in \FP{n}$ and all $X \in \varinjlim \mathcal{I}$. \label{n-coh-lim-2}
\item $\Omega^{-n+1}(\varinjlim \rm{Im}(\Psi)) \subseteq \FPinj{n}$. \label{n-coh-lim-3}
\end{enumerate}
\end{theorem}

A key result for the proof of this theorem is given in general in the Appendix \ref{Grothendieck} section and we refer to it in the proof given below.

\begin{proof}
\eqref{n-coh-lim-1} $\iff$ \eqref{n-coh-lim-2}.
Let $M \in \FP{n}$. From Corollary \ref{FP_n-dir-lim-characterization} we get that $M\in\FP{n+1}$ if and only if $\varinjlim \text{Im}(\Psi) \subseteq \text{Ker}(\text{Ext}^{n}_R(M,-))$. Since in this case $\im{\Psi} = \mathcal{I}$, then $\FP{n}\subseteq \FP{n+1}$ if and only if $\Ext^{n}_R(M,X)=0$, for all $M\in \FP{n}$ and $X \in \varinjlim \mathcal{I}$. 

\eqref{n-coh-lim-2} $\iff$ \eqref{n-coh-lim-3}. This follows from dimension shifting.
\end{proof}

A  quick application of this theorem gives the following result.
\begin{corollary}
If the injective dimension of the class $\varinjlim \text{Im}(\Psi)$ is at most $n$,  then $R$ is $n+1$-coherent. 
\end{corollary}


\section{Torsion pairs and $n$-hereditary rings} \label{sec-torsion}

As an application, we see that over $n$-hereditary ring the classes $\FPinj{n}$ and $\FPflat{n}$ define torsion classes and torsion-free classes respectively; this allows us to introduce new torsion pairs. Our approach to torsion pairs is that of B. Stenstr{\"o}m \cite{rings-of-quot}, and so is the general terminology used in the section.

\begin{definition}
A \emph{torsion pair} of a  (co)complete and locally small abelian category $\A$, is a pair $(\T,\F)$ of classes of $\A$ such that:
\begin{enumerate}
\item $\Hom_{\A}(T,F)=0$ for all $T \in \T$ and $F \in \F$.
\item If $\Hom_{\A}(C,F) = 0$ for all $F \in \F$, then $C \in \T$.
\item If $\Hom_{\A}(T,C) = 0$ for all $T \in \T$, then $C \in \F$.
\end{enumerate}
\end{definition}
In this case $\T$ is called a \emph{torsion class} and $\F$ is called a \emph{torsion-free class}. The pair $(\T,\F)$ is called  \emph{hereditary} if $\T$ is closed under subobjects. 

Given a class $\C$ of object in $\A$, we define
\[
\C^{\perp} = \{  X \in \A : \Hom_{\A}(C,X) = 0 \text{ for all } C \in \C \}  
\]
and similarly define
\[
{}^{\perp}\C = \{  X \in \A : \Hom_{\A}(X,C) = 0 \text{ for all } C \in \C \}.  
\]
This way, for any class $\C$ of $\A$, the pair $({}^{\perp}(\C^{\perp}),\C^{\perp})$ is a torsion pair.

\begin{proposition}{{\cite[Theorem VI.2.1 and Proposition VI.2.2]{rings-of-quot}}} \label{torsion-free}
Let $\T$ and $\F$ be classes of a (co)complete, locally small abelian category $\A$.
\begin{enumerate}
\item $\T$ is a torsion class for some torsion pair if and only if $\T$ is closed under quotients, coproducts and extensions.
\item $\F$ is a torsion-free class for some torsion pair if and only if $\F$ is closed under subobjects, products and extensions.
\end{enumerate}
\end{proposition}

Let $n>1$, and note that the classes $\FPinj{n}$  and $\FPflat{n}$ in $\Rmod$ are closed under direct sums and extensions. Given that $R$ is $n$-hereditary if and only if $\FPinj{n}$ is closed under quotients or $\FPflat{n}$ is closed under submodules (see Theorem \ref{teo. big}), then in combination with Proposition \ref{torsion-free} we get the following result.

\begin{theorem} \label{FPn-inj-and-FPn-flat-are-torsion-pairs}
Let $n>1$. The following statements are equivalent: 
\begin{enumerate}
\item $R$ is an $n$-hereditary ring. 
\item The pair $(\FPinj{n},\FPinj{n}^{\perp})$ is a torsion pair.  
\item The pair $({}^{\perp}\FPflat{n},\FPflat{n})$ is a torsion pair.
\end{enumerate} 
\end{theorem}

Next we give several definitions available in the literature regarding tilting and cotilting modules (see \cite{Gobel}, \cite{Handbook})

Let $T\in \Rmod$. We will say that $T$ is a \emph{$1$-tilting} $R$-module if the following assertions hold:
\begin{enumerate}
\item $T$ has projective dimension less or equal than 1.
\item $\Ext^{i}_R(T,T^{(I)})=0$, for each integer $i\geq 1$ and all sets $I$.
\item There exists an exact sequence $\ses{R}{T_0}{T_1}$ such that $T_i$ is isomorphic to a direct summands of copies of $T$, for each $i=0,1$.
\end{enumerate}

If $T$ is 1-tilting $R$-module, then the pair $(\kker{\Ext^{1}_R(T,-)},\kker{\Hom_R(T,-))}$ is a torsion pair in $\Rmod$ which is called the \emph{1-tilting torsion pair associated to $T$}, and the class $\kker{\Ext^{1}_R(T,-)}$ is called the \emph{1-tilting class associated to $T$}.

Dually, $C$ is a \emph{1-cotilting} $R$-module if it satisfies the following conditions:

\begin{enumerate}
\item $C$ has injective dimension less or equal than 1.
\item $\Ext^{i}_{R}(C^{I},C)=0$, for each integer $i\geq 1$ and all sets $I$.
\item There exists an injective cogenerator $Q$ of $\Rmod$ and there exists an exact sequence $\ses{C_1}{C_0}{Q}$ such that $C_i$ is isomorphic to a direct summands of a direct products of  copies of $C$, for each $i=0,1$.
\end{enumerate}

Similarly, the pair $(\kker{\Hom_R(-,C)} , \kker{\Ext^{1}_R(-,C))} $ is a torsion pair in $\Rmod$,  called the \emph{1-cotilting torsion pair associated to $C$} whenever $C$ is 1-cotilting $R$-module. The class $\kker{\Ext^{1}_R(-,C)}$ is called the \emph{1-cotilting class associated to $C$}.

We say that $\C$ is a 1-tilting (respectively 1-cotilting) class if there is some tilting (repectively cotilting) module $M$ such that $\C = \kker{\Ext^1_R(M,-)}$ (respectively $\C = \kker{\Ext^1_R(-,M)}$).

The next theorem follows as an application of \cite[Theorem 2.2]{AHT}. For completion we record a weak version of \cite[Theorem 2.2]{AHT} as the following proposition.

\begin{proposition} \label{AHT}
Every resolving subclass $\C$ of finitely generated modules of projective dimension at most $1$ gives rise to a 1-tilting class of $R$-modules by assigning $\C$ to $\kker{\Ext^1_R(\C,-)}$.
\end{proposition}

Here we say that a class of modules is \textit{resolving} if it  is closed under extensions, direct summands, kernels of epimorphisms in that class and contains the finitely generated projective modules. For example, the class of $\FP{\infty}$ is resolving.

\begin{theorem}
Let $n>1$. The following statements are equivalent:
\begin{enumerate}
\item $R$ an $n$-hereditary ring. \label{R-is-n-her-tilting}
\item $\FPinj{n}$ is a $1$-tilting class. \label{1-tilt}
\item $\FPflat{n}$ is a $1$-cotilting class. \label{1-cotilt}
\end{enumerate} 
\end{theorem}

\begin{proof}
\eqref{R-is-n-her-tilting} $\implies$ \eqref{1-tilt}.
Since $n>1$ and $R$ is $n$-hereditary, then by Corollary  \ref{n-her-implies-n-coh} we have that $R$ is $n$-coherent.  Now, from Proposition \ref{AHT} and Theorem \ref{characterization-of-n-coherent} we get that 
\[
\kker{\Ext^1_R(\FP{\infty},-)} = \FPinj{\infty} = \FPinj{n}
\]
 is a $1$-tilting class.

\eqref{1-tilt} $\implies$ \eqref{1-cotilt}. 
We use that $\FPinj{n}$ is a $1$-tiltitng class and apply \cite[Theorem 8.1.2]{Gobel} to get that $\FPflat{n}$ is a $1$-cotilting class.

\eqref{1-cotilt} $\implies$ \eqref{R-is-n-her-tilting}.
If $\FPflat{n}$ is a $1$-cotilting class, then it is a torsion-free class, and so it is closed under submodules. Hence by Theorem \ref{teo. big}, the ring $R$ is $n$-hereditary.
\end{proof}

We note that for the ring from Example \ref{2-her-ring} the torsion pairs from Theorem \ref{FPn-inj-and-FPn-flat-are-torsion-pairs} are not trivial. This follows from the following result.

\begin{proposition}
Let $R = \Z \oplus \bigoplus_{i \geq 1} \Z / 2\Z$ as described in Example \ref{2-her-ring}. Then we have that $\FPinj{2} \subsetneq \Rmod$ and  $\FPflat{2} \subsetneq \Rmod$.
\end{proposition}

\begin{proof}
Consider an odd integer $m \neq 1$ and the ideal $(m,a)R$. Since $(m^2,a) = (m,a)(m,a)$, then we have the following short exact sequence:
\[
\ses{(m^2,a)R}{(m,a)R}{C},
\]
where $C$ is the cokernel of the inclusion map $i: {(m^2,a)R} \to {(m,a)R} $.  We note that this short exact sequence doesn't split. In fact, if there is a map $q: (m,a)R \to (m^2,a)R$ such that $qi = id$, then we  have that 
\[
(m^2,a) = q ((m,a)(m,a)) = (m,a)q((m,a)) = (m^3n,b), 
\]
with $n$ some integer; this can't be. 

From Lemma \ref{(m,a)-is-proj} in the Appendix \ref{proof-of-the-ring}, we see that $(m,a)R$ and $(m^2,a)R$ are finitely generated projective modules, and so they are in $\FP{\infty}$, thus making $C \in \FP{\infty}$. Since $R$ is $2$-hereditary, it is also $2$-coherent and so $\FP{\infty} = \FP{2}$. Thus we have that $C \in \FP{2}$ and that $\Ext^1_R(C,(m^2,a)R) \neq 0$. Hence the $R$-module $(m^2,a)R \not \in \FPinj{2}$, giving us the first statement.

The duality between $\FPinj{n}$ and $\FPflat{n}$, gives the last statement.
\end{proof}

Furthermore if the ring is commutative, then the work of Hrbek \cite{Hrbek} allows us to say a few more results about the torsion pair associated to $\FPflat{n}$.

\begin{corollary} \label{n-her-and-comm}
Let $R$ be an $n$-hereditary and commutative ring with $n>1$.   Then we have that the torsion pair $(^{\perp}\FPflat{n},\FPflat{n})$ is an hereditary $1$-cotilting torsion pair.

\end{corollary}

\begin{proof}
This follows immediately from \cite[Proposition 3.11]{Hrbek}.
\end{proof}

\begin{remark}
As a consequence of this last result we have a far from obvious statement about the class $\FPflat{n}$. Namely, that for $n>1$ and over an $n$-hereditary ring the class $\FPflat{n}$ is closed under injective envelopes (see \cite[Proposition VI.3.2]{rings-of-quot})
\end{remark}

Corollary \ref{n-her-and-comm} also allows us to show that the torsion pair associated to $\FPflat{n}$ is a tCG torsion pair; for the definition of tCG-torsion pairs see \cite{BraPa}.

\begin{corollary}
Let $R$ be an $n$-hereditary and commutative ring with $n>1$.   Then we have that the torsion pair $(^{\perp}\FPflat{n},\FPflat{n})$ is a tCG torsion pair.
\end{corollary}

\begin{proof}
From Corollary \ref{n-her-and-comm} we have that the torsion pair is hereditary. Since the class of  $\FPflat{n}$ is closed under direct limits, then from  \cite[Corollary 3.8]{BraPa} we have the result. 
\end{proof}



\appendix
\section{Modules over the ring $\Z \oplus \bigoplus_{i \geq 1}\Z / 2\Z$} \label{proof-of-the-ring}

In this section we show some properties about the ring $R=\Z \oplus \bigoplus_{i \geq 1}\Z / 2\Z$, from Example \ref{2-her-ring}. In particular we show that it is not a semi-hereditary (1-hereditary), 2-coherent ring, that is also B\'ezout and therefore an arithmetical ring. We also include a property about the projectivity of some of its principal ideals which are of importance for the torsion pairs of the previous section. 

For notational purposes we let $A =\bigoplus_{i \geq 1}\Z / 2\Z$, thus $R = \Z \oplus A$. Recall that  addition is defined component wise,  and for $m,n \in \Z$, and for $a,b \in A$ multiplication is given by 
\[
(m,a) \cdot (n,b) = (mn, mb + na + ab),
\]
where  $ma = (m a_1,m a_2,m a_3, \ldots)$, and  $ab = (a_1b_1,a_2,b_2,a_3,b_3\ldots)$. Given $(m,a) \in R$, we define the support of $a$ as 
\[
\supp{a} = \{ i \in \Z : a_i \neq 0 \}.
\]
We begin by observing that this ring is the standard unitification of $A$, when viewed as a ring without a unit (see \cite[\S 1.1 Exercise 1]{Anderson-Fuller}). Next, we have the following technical lemma.

\begin{lemma} \label{not-in-FP1}
For any $a \in A$, the ideal $I=(2m,a)R \in  \FP{0} \setminus \FP{1}$.
\end{lemma}
\begin{proof}
Clearly $I$ is finitely generated, hence $I \in \FP{0}$. Now consider the map $R \xrightarrow{f} (2m,a)R$ given by $f((1,0)) = (2m,a)$. A quick computation shows that:
\[
\kker{f} = 0 \oplus \left( \Z / 2\Z \right)^{(\N \setminus \supp{a})},
\]
where  $\left( \Z / 2\Z \right)^{(\N \setminus \supp{a})}$ is the direct sum of $\Z / 2\Z$ in the positions given by $\N \setminus \supp{a}$ and $0$ otherwise. Note that $(0,b)R = 0 \oplus \left( \Z / 2\Z \right)^{(\supp{b})}$. Let us suppose that $\kker{f}$ is finitely generated, then we get that:
\[
 \kker{f} = (0,a_1)R + \cdots + (0,a_n)R  \subseteq 0 \oplus \left( \Z / 2\Z \right)^{(\cup \; \supp{a_i})}.
\]
This is a contradiction, since $\kker{f}=0 \oplus \left( \Z / 2\Z \right)^{(\N \setminus \supp{a})}$. Hence $\kker{f}$ is not finitely generated.
\end{proof}

\begin{proposition} \label{not-1-her}
$R$ is not 1-hereditary. 
\end{proposition}

\begin{proof}
Lemma \ref{not-in-FP1} shows a family of ideals $I \in  \FP{0} \setminus \FP{1}$; that is finitely generated ideals, that are not finitely presented. Hence $R$ is not coherent, or 1-coherent, and the result follows from from Corollary \ref{n-her-implies-n-coh}.
\end{proof}

Regarding the observation about the non-coherency of this ring we say the following.
\begin{proposition} \label{2-coherent-Bezout}
$R$ is a 2-coherent ring.
\end{proposition}
\begin{proof}
This  follows directly from Example \ref{2-her-ring} and Corollary \ref{n-her-implies-n-coh}.
\end{proof}

From {\cite[Example 1.3(b)]{Vasconcelos}} we know that the localizations of $R$ at its prime ideals are valuation domains. Thus from \cite[Section 6.4]{Faith} we have that $R$  is an \emph{arithmetical} ring (that is, a ring such that the localization at all of  its maximal ideal are valuation rings). 

However, we can say more. Indeed, we next show that  $R$ is a \emph{B\'ezout} ring, that is a ring such that  all its finitely generated ideals are principal \cite[Section 5.4B]{Faith}. 

\begin{proposition}
$R$ is a B\'ezout ring.
\end{proposition}

\begin{proof}
Let us consider the following short exact sequence:
\[
0 \to \bigoplus_{i \geq 1}\Z / 2\Z \to R \xrightarrow{\pi} \Z \to 0
\]
where $\pi(m,a) = m$. Now let $I= \langle (n_1,a_1), \cdots , (n_k,a_k) \rangle$, be a finitely generated ideal of $R$. Since $\pi(I)$ is an ideal of $\Z$, we know is a principal ideal and indeed $\pi(I) = (n_1,\ldots,n_k) = (d)$. We now split the proof in two cases depending if either $d$ is even or odd.

Suppose that $d$ is even. We consider the set $S = \bigcup_{j = 1}^n \supp{a_j}$ and define $a = \sum_{i \in S} e_i$, where $e_i$ is the element of $A$ with zero everywhere except in $i$-th position where it has a 1. Then we claim that $I = \langle (d,a) \rangle$. To see this, let  $m_i := n_i / d$ and check that.
\begin{enumerate}
\item If  $m_i$ is even, then $(n_i,a_i) = (d,a) (m_i,a_i)$.
\item If  $m_i$ is odd, then $(n_i,a_i) = (d,a) (m_i,a + a_i)$.
\end{enumerate}
Hence $I \subset \langle (d,a) \rangle$. For the other containment we just need a suitable combination of the generators giving $(d,a)$, and this is not hard to obtain. This concludes the even case.

Now, suppose that $d$ is odd. We consider the set 
\[
J = \{ a \in A : (d,a) \in I \} 
\]
and note that if $a,b \in J$, then $ab \in J$. Indeed, if $(d,a),(d,b) \in I$, then $(d,ab) = (d,a)-(d,b) + (d,a)(1,b) \in I$. 

Since $\supp{ab} \subseteq \supp{a}$, then there is $c \in J$ such that $\supp{c} \subseteq \supp{a}$, for all $a \in J$. That is $ca = c$ for all $a \in J$. We, now claim that $I = \langle (d,c) \rangle$. From the definition we have that $(d,c) \in I$, which gives the first inclusion. To see the other inclusion, let  $m_i := n_i / d$ and check the following.
 \begin{enumerate}

\item  If  $m_i$ is even, then $(d,a_i + c) = (n_i,a_i) - (d,c) (m_i-1,c) \in I$. This means that $a_i + c \in J$ and thus $c a_i = 0$. Now, 
 \[
 (n_i,a_i) = (n_i,a_i + ca_i) = (d,c)(m_i,a_i)
 \]

 \item If  $m_i$ is odd, then $(d,a_i) = (n_i,a_i) - (d,c) (m_i-1,c) \in I$. This means that $a_i \in J$ and thus $c a_i = c$, that is, $c a_i + c = 0$. Now, 
 \[
 (n_i,a_i) = (n_i,a_i + ca_i+c) = (d,c)(m_i,a_i)
 \]

\end{enumerate}
\end{proof}

We add a result about the projectivity of certain principal ideals.

\begin{lemma} \label{(m,a)-is-proj}
For any $a \in A$ and $m$ any odd integer, the ideal $(m,a)R$ is projective. 
\end{lemma}
\begin{proof}
Consider the epimorphism $R \xrightarrow{f} (m,a)R$, which sends $(1,0) \mapsto (m,a)$. Now  consider the homomorphism $(m,a)R \xrightarrow{g} R$ given by $(m,a) \mapsto (1,a)$. This is a well defined map since the equation $(m,a)(n,b) = 0$ implies that $(1,a)(n,b)=0$, when $m$ is odd. Hence any zero divisor of $(m,a)$ is also a zero divisor of $(1,a)$. Now, since $(m,a)(1,a) = (m,a)$, we quickly check that the composition $f  g (x)= x$ for any $x \in (m,a)R$, giving us a splitting of $R$.
\end{proof}


\section{A categorical result applied to finitely $n$-presented modules} \label{Grothendieck}

In this appendix we show a result in the general setting of Grothendieck categories, which when applied to  the setting of $R$-modules provides a characterization of finitely $n$-presented modules and therefore of $n$-coherent rings; namely Theorem \ref{n-coherent rings}. Although the proof can be done directly for $R$-modules, we prefer to show it in this generality. We recall some concepts of these type of categories; however, for a general treatment of Grothendieck categories, we refer the readers to \cite{rings-of-quot}.

A \emph{Grothendieck category}, $\mathcal{G}$, is a cocomplete abelian category with a generator and the direct limits are exact. It is well-known that every Grothendieck category has a \emph{injective cogenerator}, that is, an object $I$ that is injective and such that the functor $\Hom_{\mathcal{G}}(-,I)$ is faithful. Given that every Grothendieck category has products for every family of objects, we get that the faithful condition on the functor $\Hom_{\mathcal{G}}(-,I)$  is equivalent to the condition that every object of $\mathcal{G}$ is isomorphic to a subobject of a product of $I$.

Given any Grothendieck category $\mathcal{G}$ and an injective cogenerator $I$ of $\mathcal{G}$ we consider the functor $\Psi:\mathcal{G} \to \mathcal{G}$ given by  
\[
M \mapsto I^{\Hom_{\mathcal{G}}(M,I)} := \prod_{h \in \Hom_{\mathcal{G}}(M,I)}I_h,
\]
with $I_h = I$. 

Indeed, this assignment is functorial since if $f:M \to N$ is a morphism in $\mathcal{G}$, then $\Psi(f):I^{\Hom_{\mathcal{G}}(M,I)} \rightarrow I^{\Hom_{\mathcal{G}}(N,I)}$ is given by the universal property of the product in $\mathcal{G}$. That is, $\Psi(f)$ is the unique morphism in $\mathcal{G}$ such that $\pi^{N}_h \circ \Psi(f)=\pi^{M}_{h \circ f}$, where $\pi^{N}_{h}:I^{\Hom_{\mathcal{G}}(N,I)} \rightarrow I$ is the $h$-projection morphism, where $h\in \Hom_{\mathcal{G}}(N,G)$, and similarly for $\pi^{M}_{h \circ f}$. Also observe that the functor $\Psi$ comes with a natural transformation $\iota: \textrm{id}_{\mathcal{G}} \rightarrow \Psi$ which is monomorphic. 

We will denote by $\varinjlim \im{\Psi}$  the class the objects of $\mathcal{G}$ which are a direct limit of a direct system in $\im{\Psi}$.

\begin{theorem} \label{FP_n-lim-im}
Let $\mathcal{G}$ be a Grothendieck category and let $M$ be an object in $\mathcal{G}$. Consider the functor $\Psi$ described in the previous paragraph and let $n>1$. Then, 
the  functors $\Ext^{i}_{\mathcal{G}}(M,-):\mathcal{G} \to \textrm{Ab}$ preserve direct limits, for each $i=0,\ldots, n-1$, if and only if, the functors $\Ext^{i}_{\mathcal{G}}(M,-):\mathcal{G} \to \textrm{Ab}$ preserve direct limits, for each $i=0,\ldots, n-2$ and $\varinjlim \im{\Psi} \subseteq \kker{\Ext^{n-1}_{\mathcal{G}}(M,-)}$.

\end{theorem}
\begin{proof}
Let $(N_{\lambda})_{\lambda}$ be a direct system in $\im{\Psi}$ and suppose that the canonical morphism $\varinjlim \Ext^{n-1}_{\mathcal{G}}(M,N_{\lambda}) \rightarrow \Ext^{n-1}_{\mathcal{G}}(M,\varinjlim N_{\lambda})$ is an isomorphism. Since each $N_{\lambda}$ is an injective object of $\mathcal{G}$ and that $n-1>0$, then we get that $\Ext^{n-1}_{\mathcal{G}}(M,N_{\lambda}) = 0$. Therefore the previous isomorphism gives that $\Ext^{n-1}_{\mathcal{G}}(M,\varinjlim N_{\lambda}) = 0$; that is $\varinjlim N_{\lambda} \in \kker{\Ext^{n-1}_{\mathcal{G}}(M,-)}$.

For the converse, we only need to check that the functor $\Ext^{n-1}_{\mathcal{G}}(M,-):\mathcal{G} \rightarrow \textrm{Ab}$ preserves direct limits. To see this, we consider  $(M_{\lambda})_{\lambda \in \Lambda}$, a direct system in $\mathcal{G}$. Note that from the natural transformation $\iota$ of the mentioned before, we get a direct system of short exact sequences in $\mathcal{G}$ of the form:
\[
0 \rightarrow M_{\lambda} \xrightarrow{\iota_{M_\lambda}} \Psi(M_{\lambda}) \rightarrow \coker{\iota_{M_\lambda}} \rightarrow 0.
\]
Since  $\mathcal{G}$ is a Grothendieck category, we obtain the following exact sequence in $\mathcal{G}$
\[
0 \rightarrow \varinjlim M_{\lambda} \xrightarrow{\varinjlim \iota_{M_\lambda}} \varinjlim \Psi(M_{\lambda}) \rightarrow \varinjlim \coker{\iota_{M_\lambda}} \rightarrow 0.
\]
Now apply the functor $\Hom_{\mathcal{G}}(M,-):\mathcal{G} \rightarrow \text{Ab}$ to this last exact sequence and obtain  the following commutative diagram with exact rows and where $(X,Y)^i_\mathcal{G}$ denotes $\Ext^i_{\mathcal G}(X,Y)$.
\[
\begin{tikzpicture}[xscale=4]
\node (11) at (0,0) {\footnotesize	 $\varinjlim (M,\Psi(M_{\lambda}))^{n-2}_{\mathcal{G}}$};
\node (12) [right of=11,xshift=2.5cm] {\footnotesize	 $\varinjlim (M,\coker{\iota_{M_\lambda})}^{n-2}_{\mathcal{G}}$};
\node (13) [right of=12,xshift=2.25cm] {\footnotesize	 $\varinjlim (M,M_{\lambda})^{n-1}_{\mathcal{G}}$};
\node (14) [right of=13,xshift=2.0cm] {\footnotesize	 $\varinjlim (M,\Psi(M_{\lambda}))^{n-1}_{\mathcal{G}}$};
\node (21) at (0,-1.5) {\footnotesize	 $(M,\varinjlim \Psi(M_{\lambda}))^{n-2}_{\mathcal{G}}$};
\node (22) [right of= 21, xshift=2.5cm] {\footnotesize $(M, \varinjlim \coker{\iota_{M_\lambda}}^{n-2}_{\mathcal{G}}$};
\node (23) [right of=22, xshift=2.25 cm] {\footnotesize $(M,\varinjlim M_{\lambda})^{n-1}_{\mathcal{G}}$};
\node (24) [right of=23, xshift=2.0cm] {\footnotesize $(M, \varinjlim\Psi( M_{\lambda}))^{n-1}_{\mathcal{G}}$};
\draw[->] (11) -- (12);
\draw[->] (12) -- (13);
\draw[->] (13) -- (14);
\draw[->] (21) -- (22);
\draw[->] (22) -- (23);
\draw[->] (23) -- (24);
\draw[->] (11) -- node[right]{\tiny $f_1$} (21);
\draw[->] (12) -- node[right]{\tiny $f_2$} (22);
\draw[->] (13) -- node[right]{\tiny $f_3$} (23);
\draw[->] (14) --  (24);
\end{tikzpicture}
\]
%
%
%
Note that both terms on the right of each row are 0, since $\Psi(M_{\lambda})$ is an injective object of $\mathcal{G}$ and $n>1$,  and also by hypothesis  $\Ext^{n-1}_{\mathcal{G}}(M,\varinjlim \Psi(M_{\lambda}))=0$.  Now, since $f_1$ and $f_2$ are isomorphisms, then we obtain from the Snake Lemma  that $f_3$ is also an isomorphism, thus completing the proof.

\end{proof}

We apply of this last result to $\mathcal{G} = \Rmod$.

\begin{corollary} \label{FP_n-dir-lim-characterization}
Let  $M \in \Rmod$. Then $M\in \FP{n}$, if and only if, $M\in \FP{n-1}$ and $\varinjlim \im{\Psi} \subseteq \kker{\Ext^{n-1}_{\mathcal{G}}(M,-)}$.
\end{corollary}
\begin{proof}
This follows directly from the characterization of Theorem \ref{FPn_Ext_Tor} and Theorem \ref{FP_n-lim-im}.
\end{proof}


%

\bibliographystyle{alpha}
\bibliography{bibliography-n-hereditary}

\begin{thebibliography}{AHHK07}

\bibitem[AF92]{Anderson-Fuller}
Frank~W. Anderson and Kent~R. Fuller.
\newblock {\em Rings and categories of modules}, volume~13 of {\em Graduate
  Texts in Mathematics}.
\newblock Springer-Verlag, New York, second edition, 1992.

\bibitem[AHHK07]{Handbook}
Lidia Angeleri~H\"ugel, Dieter Happel, and Henning Krause, editors.
\newblock {\em Handbook of tilting theory}, volume 332 of {\em London
  Mathematical Society Lecture Note Series}.
\newblock Cambridge University Press, Cambridge, 2007.

\bibitem[AHHT06]{AHT}
Lidia Angeleri~H\"ugel, Dolors Herbera, and Jan Trlifaj.
\newblock Tilting modules and {G}orenstein rings.
\newblock {\em Forum Math.}, 18(2):211--229, 2006.

\bibitem[AS05]{Assem-Saorin}
Ibrahim Assem and Manuel Saor\'in.
\newblock Abelian exact subcategories closed under predecessors.
\newblock {\em Comm. Algebra}, 33(4):1205--1216, 2005.

\bibitem[BGH14]{BGH}
D.~{Bravo}, J.~{Gillespie}, and M.~{Hovey}.
\newblock {The stable module category of a general ring}.
\newblock {\em ArXiv e-prints}, May 2014.

\bibitem[BH09]{Bazzoni-Her}
Silvana Bazzoni and Dolors Herbera.
\newblock Cotorsion pairs generated by modules of bounded projective dimension.
\newblock {\em Israel J. Math.}, 174:119--160, 2009.

\bibitem[Bie76]{bieri}
Robert Bieri.
\newblock {\em Homological dimension of discrete groups}.
\newblock Mathematics Department, Queen Mary College, London, 1976.
\newblock Queen Mary College Mathematics Notes.

\bibitem[BP16]{BraPa}
D.~{Bravo} and C.~E. {Parra}.
\newblock {tCG Torsion Pairs}.
\newblock {\em ArXiv e-prints}, October 2016.

\bibitem[BP17]{BraPe}
Daniel Bravo and Marco~A. P\'erez.
\newblock Finiteness conditions and cotorsion pairs.
\newblock {\em J. Pure Appl. Algebra}, 221(6):1249--1267, 2017.

\bibitem[Bro82]{Brown-Cohomology-of-groups}
K.~S. Brown.
\newblock {\em Cohomology of Groups}.
\newblock Graduate Texts in Mathematics. Springer, 1982.

\bibitem[CE99]{C-E}
Henri Cartan and Samuel Eilenberg.
\newblock {\em Homological algebra}.
\newblock Princeton Landmarks in Mathematics. Princeton University Press,
  Princeton, NJ, 1999.
\newblock With an appendix by David A. Buchsbaum, Reprint of the 1956 original.

\bibitem[CGM07]{CGM}
Riccardo Colpi, Enrico Gregorio, and Francesca Mantese.
\newblock On the heart of a faithful torsion theory.
\newblock {\em J. Algebra}, 307(2):841--863, 2007.

\bibitem[Col99]{Colpi}
Riccardo Colpi.
\newblock Tilting in {G}rothendieck categories.
\newblock {\em Forum Math.}, 11(6):735--759, 1999.

\bibitem[CT95]{Colpi-Trilifaj}
Riccardo Colpi and Jan Trlifaj.
\newblock Tilting modules and tilting torsion theories.
\newblock {\em J. Algebra}, 178(2):614--634, 1995.

\bibitem[Dic66]{Dick}
Spencer~E. Dickson.
\newblock A torsion theory for {A}belian categories.
\newblock {\em Trans. Amer. Math. Soc.}, 121:223--235, 1966.

\bibitem[EJ11]{E-J}
Edgar~E. Enochs and Overtoun M.~G. Jenda.
\newblock {\em Relative homological algebra. {V}olume 1}, volume~30 of {\em De
  Gruyter Expositions in Mathematics}.
\newblock Walter de Gruyter GmbH \& Co. KG, Berlin, extended edition, 2011.

\bibitem[Fai99]{Faith}
Carl Faith.
\newblock {\em Rings and things and a fine array of twentieth century
  associative algebra}, volume~65 of {\em Mathematical Surveys and Monographs}.
\newblock American Mathematical Society, Providence, RI, 1999.

\bibitem[GT06]{Gobel}
R~G\"obel and J.~Trlifaj.
\newblock {\em Approximations and Endomorphism Algebras of Modules}.
\newblock De Gruyter Expositions in Mathematics. W. De Gruyter, 2006.

\bibitem[Hrb16]{Hrbek}
Michal Hrbek.
\newblock One-tilting classes and modules over commutative rings.
\newblock {\em J. Algebra}, 462:1--22, 2016.

\bibitem[HRS96]{HRS}
Dieter Happel, Idun Reiten, and Sverre~O. Smal\o.
\newblock Tilting in abelian categories and quasitilted algebras.
\newblock {\em Mem. Amer. Math. Soc.}, 120(575):viii+ 88, 1996.

\bibitem[Lam99]{Lam}
T.~Y. Lam.
\newblock {\em Lectures on Modules and Rings}.
\newblock Graduate Texts in Mathematics. Springer New York, 1999.

\bibitem[Meg70]{megibben}
Charles Megibben.
\newblock Absolutely pure modules.
\newblock {\em Proc. Amer. Math. Soc.}, 26:561--566, 1970.

\bibitem[PS15]{Parra-Saorin}
Carlos~E. Parra and Manuel Saor\'in.
\newblock Direct limits in the heart of a t-structure: the case of a torsion
  pair.
\newblock {\em J. Pure Appl. Algebra}, 219(9):4117--4143, 2015.

\bibitem[Roo82]{Roos}
Jan-Erik Roos.
\newblock Finiteness conditions in commutative algebra and solution of a
  problem of {V}asconcelos.
\newblock In {\em Commutative algebra: {D}urham 1981 ({D}urham, 1981)},
  volume~72 of {\em London Math. Soc. Lecture Note Ser.}, pages 179--203.
  Cambridge Univ. Press, Cambridge-New York, 1982.

\bibitem[Rot08]{rotman}
J.~J. Rotman.
\newblock {\em An Introduction to Homological Algebra}.
\newblock Universitext. Springer New York, 2008.

\bibitem[Ste70]{stenstrom}
B.~Stenstr{\"o}m.
\newblock Coherent rings and fp-injective modules.
\newblock {\em J.London Math.Soc.}, 2(2):323--329, 1970.

\bibitem[Ste75]{rings-of-quot}
Bo~Stenstr{\"o}m.
\newblock {\em Rings of quotients}.
\newblock Springer-Verlag, New York-Heidelberg, 1975.
\newblock Die Grundlehren der Mathematischen Wissenschaften, Band 217, An
  introduction to methods of ring theory.

\bibitem[Vas76]{Vasconcelos}
Wolmer~V. Vasconcelos.
\newblock {\em The rings of dimension two}.
\newblock Marcel Dekker, Inc., New York-Basel, 1976.
\newblock Lecture Notes in Pure and Applied Mathematics, Vol. 22.

\bibitem[ZP17]{Pe-Zhao}
T.~{Zhao} and M.~A. {P{\'e}rez}.
\newblock {Relative FP-injective and FP-flat complexes and their model
  structures}.
\newblock {\em ArXiv e-prints}, March 2017.

\end{thebibliography}

\end{document}